\documentclass[11pt, letterpaper, leqno]{article}
\usepackage{amsmath,amsthm,amscd,amssymb,amsbsy,cite,amstext,mathtools}
\usepackage{pgf}
\usepackage{color}
\usepackage{enumerate,enumitem}
\usepackage{mathrsfs}
\usepackage[symbol]{footmisc}
\usepackage[colorlinks=true, pdfstartview=FitV, linkcolor=blue, citecolor=orange, urlcolor=blue]{hyperref}
\usepackage{cleveref}

\usepackage{amsfonts}
\usepackage[T1]{fontenc}

\theoremstyle{plain}
\newtheorem{theorem}{Theorem}[section]
\newtheorem{lemma}{Lemma}[section]

\newtheorem{claim}{Claim}[section]

\theoremstyle{definition}
\newtheorem{definition}{Definition}[section]
\newtheorem*{acknowledgment}{Acknowledgment}
\newtheorem{problem}{Problem}
\theoremstyle{remark}
\newtheorem{remark}{Remark}[section]

\numberwithin{equation}{section}

\newcommand{\LA}[1]{\refstepcounter{equation}\text{(\theequation)}\label{#1}}

\newcommand{\LAQ}[2]{\begin{itemize}[leftmargin=1cm]\item[\LA{#1}]{#2} \end{itemize}}
\numberwithin{equation}{section}

\newcommand{\void}{\varnothing}
\renewcommand{\subset}{\subseteq}
\newcommand{\abs}[1]{\left|#1\right|}
\newcommand{\R}{{\mathbb{R}}}

\newcommand{\Rn}{{\mathbb{R}^n}}

\newcommand{\set}[1]{\left\{#1\right\}}

\newcommand{\bra}[1]{\left(#1\right)}
\newcommand{\eps}{\epsilon}
\newcommand{\supp}{\mathrm{supp}}
\newcommand{\dist}{\mathrm{dist}}
\newcommand{\diam}{\mathrm{diam}}
\newcommand{\qt}[1]{``#1''}

\begin{document}

	\title{Nonnegative Whitney Extension Problem for $ C^1(\mathbb{R}^n) $}
	\author{Fushuai Jiang}
	
	\maketitle
	
	\newcommand{\jet}{J}
	\newcommand{\Rj}{\mathcal{R}}
	\newcommand{\G}{\Gamma}
	\renewcommand{\P}{\mathcal{P}}
	\newcommand{\ksh}{{k^{\sharp}}}
	\newcommand{\jm}{\odot}
	\newcommand{\m}{\mathfrak{m}}

	\begin{abstract}
		Let $ f $ be a real-valued function on a compact subset in $ \mathbb{R}^n $. We show how to decide if $ f $ extends to a nonnegative and $ C^1 $ function on $ \mathbb{R}^n $. There has been no known result for nonnegative $ C^m $ extension from a general compact set when $ m > 0 $. The nonnegative $ C^m $-extension problem for $ m \geq 2 $ remains open.
	\end{abstract}

	\section{Introduction}

	For $ m,n \geq 1 $, we write $ C^m(\Rn) $ to denote the vector space of continuously differentiable functions on $ \Rn $ whose derivatives up to $ m $-th order are bounded and continuous. We write $ C^m_+(\Rn) $ to denote the collection of elements in $ C^m(\Rn) $ that are also nonnegative on $ \Rn $.

	In this paper, we consider the following problem. 
	
	\begin{problem}[Nonnegative Whitney Extension Problem]\label{prob.W+}
		Let $ E \subset \Rn $ be compact. Let $ f : E \to [0,\infty) $. How can we decide if there exists $ F \in C^m_+(\Rn) $ with $ F = f $ on $ E $?
	\end{problem}

	
	When $ E $ is finite, \cite{JL20,FIL16+} provide solutions to Problem \ref{prob.W+} with further control on the size of the derivatives of the extension (an extension without derivative control always exists in this case). It is related to the $ C^m $ selection problem. However, when $ E $ is infinite, the strategies employed in \cite{JL20,FIL16+} collapse, because they rely on a Calder\'on-Zygmund decomposition procedure which may not terminate when $ E $ is infinite. There has been no known answer to {Problem \ref{prob.W+}} when $ E \subset \Rn $ is infinite.

	Problem \ref{prob.W+} is a variant the following classical problem posed by H. Whitney\cite{W34-1,W34-2,W34-3}.
	
	\begin{problem}[Whitney Extension Problem]\label{prob.W}
			Let $ E \subset \Rn $ be compact. Let $ f : E \to \R $. How can we decide if there exists $ F \in C^m(\Rn) $ with $ F = f $ on $ E $?
	\end{problem}
	
	In a series of papers \cite{F05-Sh,F05-J,F06}, C. Fefferman provided a solution to Problem \ref{prob.W}. A key ingredient in Fefferman's solution is the notion of ``Glaeser refinement'', inspired by \cite{G58,BMP03}. We briefly discuss the main idea of \cite{F06} here. 
	
	To each $ x \in E $, we assign an affine subspace $ H_f(x) \subset \P^m $, where $ \P^m $ denotes the polynomial of $ n $ variables of degree no greater than $ m $. 
	The subspace $ H_f(x) $ satisfies the following crucial property:
	\LAQ{}{If $ F \in C^m(\Rn) $ satisfies $ F = f $ on $ E $, then $ \jet_x^mF \in H_f(x) $.}
	Here, $ \jet_x^mF $ denotes the degree $ m $ Taylor polynomial of $ F $ about the point $ x $. For instance, we may take $ H_f(x) = \set{P \in \P^m : P(x) = f(x)} $.
	Then solving Problem \ref{prob.W} then amounts to the following problem.
	\LAQ{}{Decide if there exists $ F \in C^m(\Rn) $ such that $ \jet_x^mF \in H_f(x) $ for all $ x \in E $.}
	
	To achieve this goal, the author uses the procedure called \qt{Glaeser refinement} (See Definition \ref{def.GR}) on each of the subspace $ H_f(x) $, which produce another subspace $ \widetilde{H}_f(x) \subset H_f(x) \subset \P^m_n $ that possibly excludes some jets at $ x $ that {\em cannot} arise as the jets of a $ C^m $ function that agrees with $ f $ on $ E $. The author first shows that the Glaeser refinement stabilizes (i.e. the procedure does not produce new proper subspace) after a controlled number (depending only on $ m $ and $ n $) of times. The author then shows that if the stabilized subspace is nonempty for each $ x \in E $, then there exists $ F \in C^m(\Rn) $ with $ \jet_x^m F \in H_f(x) $ for all $ x \in E $, hence solving Problem \ref{prob.W}.
	
	In this paper, we adapt the technology described above to solve Problem \ref{prob.W+} for $ m = 1 $ (see Theorem \ref{thm.main} in Section \ref{sec.PMR}). To account for nonnegativity, we associate to each $ x \in E $ a subset
	\begin{equation*}
	\G_f(x) = \set{P \in \P^1 : \text{ there exists }F \in C^1_+(\Rn) \text{ such that } F(x) = f(x) \text{ and } \jet_x^1F = P }\,.
	\end{equation*}
	Solving Problem \ref{prob.W+} then amounts to deciding whether there exists $ F \in C^1_+(\Rn) $ such that $ \jet_x^1F \in \G_f(x) $ for each $ x \in E $. 
	
	To this end, we will apply Glaeser refinement to each of the subset $ \G_f(x) $. Following \cite{F06}, we will first prove that each subset $ \G_f(x) $ will eventually stabilize after a finite number of refinement. Next, we show that if, for each $ x \in E $, we start with $ \G_f(x) $ and arrive at some $ \G_*(x) \neq \void $ after a certain number of refinement, and that $ \G_*(x) $ is its own Glaeser refinement; then there exists $ F \in C^1_+(\Rn) $ such that $ \jet_x^1 F \in \G_f(x,\infty) $ for each $ x \in E $, hence solving Problem \ref{prob.W+} for $ m = 1 $. 
	
	This paper is part of a literature on extension and interpolation, going back to the seminal works of H. Whitney \cite{W34-1,W34-2,W34-3}. We refer to the interested readers to \cite{F05-Sh,F05-J,F06,FIL16,FIL16+} and references therein for the history and related problems. For further discussion on Glaeser refinement, we direct the readers to \cite{BMP06,BM07,KZ07}.

	\begin{acknowledgment}
		I am indebted to Kevin Luli for introducing me to this fascinating subject, and for his constant guidance and valuable suggestion. I thank Kevin O'Neill for his useful suggestion. I would also like to thank the participants of NSF-CBMS Conference: Fitting Smooth Functions to Data for valuable discussions.
	\end{acknowledgment}
	
	We will start from scratch and redefine all the notions.

	\section{Preliminaries and Main Results}
	\label{sec.PMR}
	Fix integers $ m,n \geq 0 $. 
	\begin{itemize}
		\item We will use Euclidean distance $ \abs{\cdot} $ on $ \Rn $. We use $ B(x,r) $ to denote the open ball of radius $ r $ centered at $ x $. 
		\item We use $ C^m_{loc}(\Rn) $ to denote  the vector space of $ m $-times continuously differentiable functions on $ \Rn $. We use $ C^m(\Rn) $ to denote the subspace of $ C^m_{loc}(\Rn) $ consisting of elements whose derivatives up to $ m $-th order are bounded on $ \Rn $. We use $ C^m_+(\Rn) $ to denote the convex subcollection of elements in $ C^m(\Rn) $ that are also nonnegative on $ \Rn $. 
		\item We use $ \P^m $ to denote the space of polynomials of $ n $ variables and degree less or equal to $ m $. For $ x \in \Rn $ and $ F \in C^m_{loc}(\Rn) $, we use $ \jet_x^mF $ to denote the $ m $-jet of $ F $ at $ x $, which we identify with the degree $ m $ Taylor polynomial of $ F $ at $ x $
		\begin{equation*}
		\jet_x^m F(y) := \sum_{\abs{\alpha}\leq m}
		\frac{\partial^\alpha F(x)}{\alpha !}(y-x)^\alpha\,.
		\end{equation*}
		We use $ \Rj_x^m $ to denote the ring of $ m $-jets at $ x $. It is clear that $ \Rj_x^m $ is isomorphic to $ \P^m $ as vector spaces, but we will distinguish them. 
		Let $ P,P' \in \Rj_x^m $, we define the jet product of $ P $ and $ P' $ in $ \Rj_x^m $ to be 
		\begin{equation*}
		P\odot_x^m P' := \jet_x^m (PP')\,.
		\end{equation*}
		
		\item We assume that $ \ksh $ is a sufficiently large integer depending only on $ m $ and $ n $. See \cite{F06} for an estimate of the size of $ \ksh $.
	\end{itemize}

	We first define the notion of Glaeser refinement.
	
	\begin{definition}\label{def.GR}
		Let $ E \subset \Rn $ be compact. For each $ x \in E $, suppose we are given a subset (not necessarily affine and possibly empty) $ \Phi_0(x) \subset \Rj_x^m $. For $ \ell \geq 0 $, we define each $ \Phi_{\ell+1}(x) $ inductively:
		
		Let $ x_0 \in E $, $ P_0 \in \Rj_{x_0}^m $, and $ \ell \geq 0 $, we say that $ P_0 \in \Phi_{\ell+1}(x_0) $ if
		\LAQ{GR}{given $ \eps > 0 $, there exists $ \delta > 0 $ such that for any $ x_1, \cdots, x_{\ksh} \in B(x_0,\delta) $, there exist $ P_1, \cdots, P_\ksh \in \P $, with $ P_j \in \Phi_{\ell}(x_j) $ for $ j = 0, \cdots, \ksh $, such that 
			\begin{equation*}
			\abs{\partial^\alpha \bra{P_i - P_j}(x_i)} \leq \eps \abs{x_i - x_j}^{m-\abs{\alpha}}
			\text{ for } \abs{\alpha} \leq m
			\text{ and }
			0 \leq i,j \leq \ksh\,.
			\end{equation*}}
		 We define the \underline{Glaeser refinement of $ \Phi_{\ell}(x_0) $} to be $ \Phi_{\ell+1}(x_0) $.
	\end{definition}

	\begin{remark}
		Without further assumption on $ \Phi_0(x) $, we do not know if $ \Phi_0(x) $ will stabilize after finite number of Glaeser refinement, i.e., if $ \Phi_{\ell^*+1}(x) = \Phi_{\ell^*}(x) $ for some $ \ell^* < \infty $. 
	\end{remark}
	
	We make a definition for a class of subsets of $ \P^m $ that will be known to stabilize.
	
	\begin{definition}\label{def.GF}
		Let $ E \subset \Rn $ be compact. For $ x \in E $, let $ \Phi(x) \subset \Rj^m_x $. We call $ \Phi(x) $ a \underline{Glaeser fiber} if $ \Phi(x) = \void $ or $ \Phi(x) $ has the form 
		\begin{equation*}
		\Phi(x) = P^x + I(x)
		\end{equation*}
		where $ P^x \in \Rj_x^m $ and $ I(x) \subset \Rj_x^m $ is an ideal. 
	\end{definition}

	The main theorem in \cite{F06} that provides an answer to Problem \ref{prob.W} is the following.
	
	\begin{theorem}[\hspace{1sp}\cite{F06}]\label{thm.Fef06}
		Let $ E \subset \Rn $ be compact. Suppose that for each $ x \in E $, we are given a nonempty Glaeser fiber $ \Phi_*(x) \subset \Rj_x^m $. Assume that $ \Phi_*(x) $ is its own Glaeser refinement. Then there exists $ F \in C^m(\Rn) $ with $ \jet_x^m F \in \Phi_*(x) $ for all $ x \in E $. 
	\end{theorem}
	
	We explain how we go from Theorem \ref{thm.Fef06} to answer Problem \ref{prob.W}. We begin with 
	\begin{equation}
	\Phi_0(x) := H_f(x): = \set{P \in \P^m_n : P(x) = f(x)}.
	\label{eq.phi0}
	\end{equation}
	Then $ \Phi_0(x) $ has the form $ f(x) + \m_0(x) $, where $ f(x) $ is the constant polynomial and 
	\begin{equation}
	\m_0(x) := \set{\phi^x \in \P^m_n: \text{ There exists } F \in C^m(\Rn)
		\text{ such that } F(x) = 0 \text{ and }
		\jet_x^m F = \phi^x
	}
	\label{m0}
	\end{equation}
	is clearly an ideal in $ \Rj_x^m $. Lemma 2.1 in \cite{F06} shows that $ \Phi_{\ell}(x) $ is still a Glaeser fiber under Glaeser refinement if we start with \eqref{eq.phi0}. Lemma 2.2 in \cite{F06} then shows that with this choice of $ \Phi_0 $, we have $ \Phi_\ell(x) = \Phi_{\ell^*}(x) $ for all $ \ell \geq \ell^* $, where $ \ell^* = 2\dim\P^m + 1 $. Therefore, deciding whether $ f : E \to \R $ extends to a $ C^m $ function amounts to computing the $ \ell^* $-th Glaeser refinement of $ H_f $. 
	\\

	Now we describe the key objects in this paper that are analogous to those above but also take into consideration of nonnegativity.

	\begin{definition}\label{def.G}
		Let $ E \subset \Rn $ be a compact subset. Let $ f : E \to [0,\infty) $. For $ x \in E $ and $ M > 0 $, we define
		\begin{equation*}
		\G_f^{(m)}(x) := \set{P \in \P^m : \text{ There exists } F \in C^m_+(\Rn) \text{ and } \jet_x F = P}\,.
		\end{equation*}
	\end{definition}

	\begin{remark}
		$ \G_f^{(m)}(x) $ is in general not a Glaeser fiber if $ m \geq 2 $. However, for $ m = 1 $, we will see in Lemma \ref{lem.GH} that it is. We will also see in Lemma \ref{lem.ideal} that it remains Glaeser after refinement.
	\end{remark}

	Our main result of the paper is the following. 
	
	\begin{theorem}\label{thm.main}
		Let $ m = 1 $. Let $ E \subset \Rn $ be compact, and let $ f : E \to [0,\infty) $ be given. For each $ x \in E $, let $ \Phi_0(x) := \G_f^{(1)}(x) $, and  for $ \ell \geq 0 $, let $ \Phi_{\ell+1}(x) $ be the Glaeser refinement of $ \Phi_{\ell}(x) $ defined by \eqref{GR}. 
		
		If $ \Phi_{2n+3}(x) \neq \void $ for each $ x \in E $, then there exists $ F \in C^1_+(\Rn) $ such that $ \jet_x^1 F \in \G_{f}^{(1)}(x) $ for each $ x \in E $. In particular, there exists $ F \in C^1_+(\Rn) $ such that $ F = f $ on $ E $.
	\end{theorem}

	To prove Theorem \ref{thm.main}, we will show that under its hypotheses, the hypotheses of Theorem \ref{thm.Fef06} (with $ m = 1 $) are satisfied. Theorem \ref{thm.Fef06} then produces a $ C^1 $ function, which is not necessarily nonnegative, and whose jet at each $ x \in E $ belongs to $ \G_{f}^{(1)}(x) $. We will then use these jets to reconstruct a nonnegative counterpart that takes the same jet at each $ x \in E $, hence solving Problem \ref{prob.W+}. The reconstruction uses a variant of the classical Whitney Extension Theorem.

	\section{Main ingredients}
	
	In this section, we prove the main ingredients.

	\subsection{Preservation of Glaeser fiber}

	The main result we prove in this subsection is the following lemma, which states Glaeser fiber remains Glaeser after refinement.

	\begin{lemma}\label{lem.glaeser form}
		Suppose $ \Phi_\ell(x) $ is a Glaeser fiber for each $ x \in E $, then $ \Phi_{\ell+1}(x) $ is a Glaeser fiber for each $ x \in E $. 
	\end{lemma}


	\begin{proof}
		We expand the argument given by Lemma 2.1 in \cite{F06}.
		
		Fix $ x_0 \in E $. If $ \Phi_{\ell+1}(x_0) = \void $, there is nothing to prove. 
		
		Suppose $ \Phi_{\ell+1}(x_0) \neq \void $. Pick arbitrary $ P^{x_0}_{\ell+1} \in \Phi_{\ell+1}(x_0) $. Let
		\begin{equation}
		I_{\ell+1}(x_0) := \Phi_{\ell+1}(x_0) - P^{x_0}_{\ell+1}. 
		\label{eq.Iel+1}
		\end{equation}
		To show that $ \Phi_{\ell+1}(x_0) $ is a Glaeser fiber, it suffices to show that $ I_{\ell+1}(x_0) $ is an ideal in $ \Rj_{x_0}^m $. 
		
		By assumption, for each $ x \in E $,
		\begin{equation}
		\Phi_{\ell}(x) = P_\ell^x + I_\ell(x)\,,
		\end{equation}
		where $ P_\ell^x \in \Rj^m_x $, and $ I_\ell(x) \subset \Rj^m_x $ is an ideal.
		
		\begin{claim}\label{claim.homoge}
			$ I_{\ell+1}(x_0) $ are defined by the following procedure:
			\LAQ{form1}{$ \phi_0 \in I_{\ell+1}(x_0) $
				if and only if the following holds: given $ \eps > 0 $, there exists $ \delta > 0 $ such that for any $ x_1, \cdots, x_\ksh \in E \cap B(x_0,\delta) $, there exist $ \phi_1, \cdots, \phi_\ksh $, with $ \phi_j \in I_\ell(x_j) $ for $ j = 0, \cdots, \ksh $, such that
			}
		\begin{equation}
		\abs{\partial^\alpha \bra{\phi_i - \phi_j}(x_i)} \leq \eps \abs{x_i - x_j}^{m-\abs{\alpha}}
		\text{ for }
		\abs{\alpha} \leq m, \,0 \leq i,j \leq \ksh\,.
		\label{eq.glaeser compatible}
		\end{equation}
		\end{claim}

		\begin{proof}[Proof of Claim \ref{claim.homoge}]
			
			First we show sufficiency. Suppose $ \phi_0 \in I_{\ell+1}(x_0) $. Fix $ \eps_0 > 0 $. Write $ P_0 = P^{x_0}_{\ell+1} $. Define
			\begin{equation*}
			\hat{P}_0 := P_0 + \phi_0\,.
			\end{equation*}
			By \eqref{eq.Iel+1}, 
			\begin{equation*}
			\hat{P}_0 \in \Phi_{\ell+1}(x_0)\,.
			\end{equation*}
			Applying Definition \ref{def.GR} to $ P_0 $ and $ \hat{P}_0 $ with $ \frac{1}{2}\eps_0 $ in place of $ \eps $, we find a $ \delta > 0 $ such that for any $ x_1, \cdots, x_{\ksh} \in E \cap B(x_0,\delta) $, there exist $ P_j, \hat{P}_j \in \Phi_{\ell}(x_i) $ for $ j = 1, \cdots, \ksh $, such that 
			\begin{equation}
				\begin{split}
				\abs{\partial^\alpha (P_i - P_j)(x_i)} &\leq \frac{1}{2}\eps_0\abs{x_i - x_j}^{m-\abs{\alpha}}\,,\\
				\abs{\partial^\alpha (\hat{P}_i - \hat{P}_j)(x_i)} &\leq \frac{1}{2}\eps_0\abs{x_i - x_j}^{m-\abs{\alpha}}\,.
				\end{split}
				\label{eq.compat}
			\end{equation}
			Now, let
			\begin{equation*}
			\phi_j:= \hat{P}_j - P_j
			\text{ for } j = 1, \cdots, \ksh\,.
			\end{equation*}
			Thanks to \eqref{eq.compat}, the $ \phi_j $'s satisfy \eqref{eq.glaeser compatible}. 
			
			Now we need to check that $ \phi_j \in I_\ell(x_j) $ for $ j = 1, \cdots, \ksh $. Indeed, by the induction hypothesis, the $ \Phi_{\ell}(x_i) $'s are Glaeser fiber, so 
			\begin{equation*}
			\phi_j \in I_\ell(x_j)
			\text{ for } j = 1, \cdots, \ksh\,.
			\end{equation*}
			This proves sufficiency.
			
			Now we show necessity. Let $ \eps_0 > 0 $ be given. Apply Definition \ref{def.GR} to $ P_0 $ and apply the latter condition in \eqref{form1} to $ \phi_0 $, with $ \frac{1}{2}\eps_0 $ in place of $ \eps $, we see that there exists $ \delta > 0 $ such that for any $ x_1, \cdots, x_\ksh \in E \cap B(x_0,\delta) $, there exist $ P_j \in \Phi_{\ell}(x_j) $ and $ \phi_j \in I_\ell(x_j) $, $ j = 1, \cdots, \ksh $, satisfying
			\begin{equation}
			\begin{split}
			\abs{\partial^\alpha (P_i - P_j)(x_i)} &\leq \frac{1}{2}\eps_0\abs{x_i - x_j}^{m-\abs{\alpha}}\,,\\
			\abs{\partial^\alpha (\phi_i - \phi_j)(x_i)} &\leq \frac{1}{2}\eps_0\abs{x_i - x_j}^{m-\abs{\alpha}}\,.
			\end{split}
			\label{eq.compat2}
			\end{equation}
			Now, let
			\begin{equation}
			\hat{P}_j := P_j + \phi_j
			\text{ for } j = 1, \cdots, \ksh\,.
			\label{eq.defPhat}
			\end{equation}
			By induction hypothesis, the $ \Phi_{\ell}(x_j) $'s are Glaeser, so
			\begin{equation*}
			\hat{P}_j \in \Phi_{\ell}(x_j)
			\text{ for all } j = 1, \cdots, \ksh\,.
			\end{equation*}
			Thanks to \eqref{eq.compat2}, we have
			\begin{equation*}
			\abs{\partial^\alpha (\hat{P}_i - \hat{P}_j)(x_i)} \leq \eps_0\abs{x_i - x_j}^{m-\abs{\alpha}}
			\text{ for } 0 \leq i,j \leq \ksh\,.
			\end{equation*}
			Therefore,
			\begin{equation}
			\hat{P}_0 \in \Phi_{\ell+1}(x_0)\,.
			\label{eq.ah}
			\end{equation}
			Now, \eqref{eq.Iel+1}, \eqref{eq.defPhat}, and \eqref{eq.ah} together imply $ \phi_0 \in I_{\ell+1}(x_0) $. This proves necessity, and concludes the proof of the claim.
		\end{proof}

		To finish the proof of the lemma, we fix $ \phi_0 \in I_{\ell+1}(x_0) $ and $ \tau \in \Rj_{x_0}^m $. It suffices to show that
		\begin{equation}
		\tilde{\phi}_0 := \phi_0 \odot_{x_0}^m \tau \in I_{\ell+1}(x_0)\,.
		\label{ideal conc}
		\end{equation}
		Let $ \eps_0 > 0 $. Let $ \delta, x_1, \cdots, x_{\ksh}, \phi_1, \cdots, \phi_{\ksh} $ be as in \eqref{form1} with $ A^{-1}\eps_0 $ in place of $ \eps $, for some $ A > 0 $ to be determined. Define
		\begin{equation*}
		\tilde{\phi}_j := \phi_j \odot_{x_j}^m \tau 
		\text{ for } j = 1, \cdots, \ksh\,.
		\end{equation*}
		Since $ I_\ell(x_j) \subset \Rj_{x_j}^m $ is an ideal by assumption, we have $ \tilde{\phi}_j \in I_\ell(x_j) $ for all $ j = 1, \cdots, \ksh $. 
		Moreover, by the classical Whitney Extension Theorem for finite set (see e.g. \cite{St79}) and \eqref{eq.glaeser compatible}, for each distinct pair $ x_i, x_j $, we may find $ F^{ij} \in C^{m}(\Rn) $ such that
		\begin{equation}
		\abs{\partial^\alpha F} \leq M
		\text{ for }
		\abs{\alpha} \leq m
		\text{ on }
		\Rn
		\label{eq.TC1}
		\end{equation}
		where $ M $ is a number depending only on $ m,n $, and $ \eps_0 $, and that
		\begin{equation}
			\jet_{x_\nu}^m F^{ij} = \phi_\nu
			\text{ for }\nu = i,j\,.
			\label{eq.TC2}
		\end{equation}
		Therefore, $ \tilde{\phi}_\nu = \jet_{x_\nu}^m(F^{ij}\cdot\tau) $ for $ \nu = i,j $. Taylor's theorem, combined with \eqref{eq.TC1} and \eqref{eq.TC2}, implies
		\begin{equation*}
		\abs{\partial^\alpha \bra{\tilde{\phi}_i - \tilde{\phi}_j}(x_i)}
		= \abs{\partial^\alpha \bra{F^{ij}\cdot\tau - \jet_{x_j}^m(F^{ij}\cdot\tau)}(x_j)} \leq B_\tau\cdot A^{-1}\eps_0\abs{x_i-x_j}^{m-\abs{\alpha}}
		\text{ for }
		\abs{\alpha}\leq m\,.
		\end{equation*}
		Here, we may take $ B_\tau $ to be a number that depends only on $ M $ and $ \tau $. Taking $ A > B_\tau $, we can conclude that
		\begin{equation*}
		\abs{\partial^\alpha \bra{\tilde{\phi}_i - \tilde{\phi}_j}(x_i)} \leq \eps_0\abs{x_i - x_j}^{m-\abs{\alpha}}
		\text{ for all }
		\abs{\alpha} \leq m, \, 0 \leq i,j \leq \ksh\,.
		\end{equation*}
		Hence, we have shown \eqref{ideal conc}. The lemma is proved. 
	\end{proof}

	\begin{lemma}[]\label{lem.stabPhi}
		Suppose $ \Phi_\ell(x) \subset \Rj_x^m $ is a Glaeser fiber for each $ x \in E $ and $ \ell \geq 0 $. If $ \ell^* = 2\dim \P^m + 1 $, then for each $ x \in E $,  $ \Phi_\ell(x) = \Phi_{\ell^*}(x) $ for all $ \ell \geq \ell^* $. 
	\end{lemma} 
	
	The argument is the same that of Lemma 2.2 in \cite{F06}, which is inspired by \cite{BMP03} and \cite{G58}. We direct the interested readers to those cited above as well as \cite{KZ07} for a discussion on stabilization of Glaeser refinement.

	\subsection{Nonnegative Whitney Extension Theorem}
	
	In this subsection, we sketch the proof of the nonnegative version of the classical Whitney Extension Theorem\cite{W34-1}.

	\begin{theorem}\label{thm.nw}
		Let $ E \subset \Rn $ be compact. Let $ f: E \to [0,\infty) $. Let $ \set{P^x : x \in E} $ be a collection of polynomials such that $ P^x \in \G_f^{(m)}(x) $ for all $ x \in E $.
		Suppose
		\begin{equation}
		\abs{\partial^\alpha \bra{P^x - P^y}(x)} = o\bra{\abs{x-y}^{m-\abs{\alpha}}}
		\text{ as } \abs{x-y} \to 0,\, \text{ for all }x, y \in E
		\text{ and } \abs{\alpha} \leq m\,.
		\label{eq.wc}
		\end{equation}
		Then there exists $ F \in C^m_+(\Rn) $ with $ \jet_x^mF = P^x $ for all $ x \in E $. 
	\end{theorem}

	\newcommand{\W}{\mathcal{W}}
	\renewcommand{\int}{\mathrm{interior}}
	
	\begin{proof}[Sketch of Proof]
		Let $ \W_E $ be a Whitney cover of $ \Rn\setminus E $, namely, $ \W_E = \set{Q}_{Q \in \W_E} $ such that the following hold.
		\begin{itemize}
			\item Each $ Q \in \W_E $ is a closed cube in $ \Rn $.
			\item If $ Q, Q' \in \W_E $ and $ Q \neq Q' $, then $ \int(Q) \cap \int(Q') = \void $.
			\item for every $ Q \in \W_E $,
			\begin{equation*}
			\frac{1}{4} \diam (Q) \leq \dist(Q,E) \leq 4 \diam(Q)\,.
			\end{equation*}
			
		\end{itemize}
		Let $ \set{\varphi_Q: Q \in \W_E} $ be a $ C^m $ partition of unity satisfying
		\begin{itemize}
			\item $\sum_{Q \in \W_E}\varphi_Q(x) = 1$ for all $ x \in \Rn\setminus E $.
			\item $ \supp(\varphi_Q) \subset \frac{3}{2} Q $. 
			\item $ \abs{\partial^\alpha\varphi_Q} \leq C(\diam Q)^{-\abs{\alpha}} $ for all $ \abs{\alpha} \leq m $.
		\end{itemize}
		For the existence of such covering and partition of unity, see e.g. \cite{St79,W34-1}.

		For each $ x \in E $, since $ P^x \in \G_f(x) $, there exists $ F^x \in C^m_+(\Rn) $ such that $ \jet_x^mF^x = P^x $. 
		
		For each $ Q \in \W_E $, we pick a representative point $ r_Q \in E $ (not necessarily unique) such that
		\begin{equation*}
		\dist(r_Q, Q) = \dist(E,Q). 
		\end{equation*}
		
		We also let 
		\begin{equation*}
		F_Q := F^{r_Q}\,.
		\end{equation*}

		Define
		\begin{equation*}
		F(x):= \left\{\begin{matrix}
		\sum\limits_{Q \in \W_E}\varphi_Q(x)F_Q(x)
		& x \in   \Rn\setminus E\\
		\,\,\,\,\,\,\,\,\,\,\,f(x)
		&x \in   E
		\end{matrix}\right.\,\,.
		\end{equation*}

		We want to show that $ F \in C^m(\Rn) $ with $ F \geq 0 $ on $ \Rn $, and $ \jet^m_x F = P^x $ for all $ x \in E $. 
		
		It is clear that $ F \geq 0 $ on $ \Rn $, since all of the $ \varphi_Q $'s and the $ F_Q $'s are.
		
		It is also clear that $ F $ is $ C^m $ away from $ E $ since each of the $ F_Q $'s are and the supports of the $ \varphi_Q $'s have bounded overlap. Therefore, it suffices to examine the differentiability property of $ F $ near the set $ E $ and the jet of $ F $ on $ E $. 
		
		By Taylor's theorem, 
		\begin{equation*}
		\partial^\alpha F_Q(x) = \partial^\alpha P^{r_Q}(x) + o\bra{\abs{x-r_Q}^{m-\abs{\alpha}}}
		\text{ as } x \to r_Q
		\text{ for all } \abs{\alpha} \leq m\,.
		\end{equation*}
		The compatibility condition \eqref{eq.wc} then implies that 
		\begin{equation*}
		\partial^\alpha F_Q(\hat{x}) \to \partial^\alpha P^{\hat{x}}(\hat{x})
		\end{equation*}
		uniformly along any sequence of cubes $ Q \in \W_E $ converging to $ \hat{x} \in E $\footnote[2]{Here we define $ \dist(x,F) = \inf\set{\abs{x - y} : y \in F} $ for $ x \in \Rn $ and $ F \subset \Rn $ closed.}. Therefore, $ F $ is $ C^m_{loc} $ near $ E $ and $ \jet^m_x F = P^x $ for each $ x \in E $. Since $ E $ is compact, we can conclude that $ F \in C^m(\Rn) $.

		This completes the sketch of the proof.

	\end{proof}

	\subsection{Properties of $ \G_{\ell} $}

	Recall Definition \ref{def.G} and \eqref{eq.phi0}. For the rest of this section, we fix $ m = 1 $. We write $ \P $ for $ \P^1 $, $ \jet_x $ for $ \jet_x^1 $, and $ \G_f(x) $ for $  \G_f^{(1)}(x) $.

	\begin{lemma}\label{lem.GH}
		If $ f(x) > 0 $, then $ \G_f(x) = H_f(x) $. If $ f(x) = 0 $, then $ \G_f(x) = \set{0} $. In particular, for each $ x \in E $, $ \G_f(x) $ is a Glaeser fiber (see Definition \ref{def.GF}).
	\end{lemma}

	\begin{proof}
		Suppose $ f(x) > 0 $. It is clear that $ \G_f(x) \subset H_f(x) $. It suffices to show the reverse inclusion. Let $ P \in H_f(x) $. Then $ P(x) = f(x) > 0 $. Since $ P $ is continuous, there exists $ \delta > 0 $ such that $ P \geq 0 $ on $ B(x,\delta) $. Let $ \chi $ be a $ C^1_+ $-cutoff function such that $ \chi \equiv 1 $ near $ x $ and $ \supp (\chi) \subset B(x,\delta) $. Then $ F := \chi \cdot P \in C^1_+(\Rn) $ with $ \jet_x F = P $. Therefore, $ H_f(x) \subset \G_f(x) $. 
		
		Suppose $ f(x) = 0 $. It is clear that the zero polynomial $ 0 \in \G_f(x) $. Suppose $ P \in \G_f(x) $, then there exists $ F \in C^1_+(\Rn) $ such that $ \jet_x F = P $. Since $ F \geq 0 $ on $ \Rn $, $ F $ has a local minimum at $ x $, so $ \nabla F(x) = 0 $. Hence, $ P \equiv 0 $. 
	\end{proof}

	\begin{lemma}\label{lem.ideal}
		Let $ \ell \geq 0 $. For each $ x \in E $, $ \G_{\ell}(x) $ is a Glaeser fiber.
	\end{lemma}

	\begin{proof}
		We have shown in Lemma \ref{lem.GH} that $ \G_0(x) $ is a Glaeser fiber for each $ x \in E $. Therefore, the Lemma follows from Lemma \ref{lem.glaeser form}.
	\end{proof}

	\begin{remark}
		A subtle difference between Lemma \ref{lem.ideal} and Lemma 2.1 in \cite{F06} is that $ \G_l(x) $ is a translate of an ideal that possibly depends on the function $ f $.
	\end{remark}

	\begin{lemma}\label{lem.stab}
		For each $ x \in E $, $ \G_{\ell^*}(x) = \G_{2n+3}(x) $ for all $ \ell^* \geq 2n+3 $.
	\end{lemma}

	\begin{proof}
		By Lemma \ref{lem.ideal}, $ \G_{\ell}(x) $ is a Glaeser fiber for each $ x \in E $ and $ \ell \geq 0 $. Therefore, by Lemma \ref{lem.stabPhi}, $ \G_{\ell^*}(x) $ is a stabilized Glaeser fiber if $ \ell^* \geq 2\dim\P + 1 = 2n+3 $.
	\end{proof}

	\section{Proof of the main theorem}

	In this section, we fix $ m = 1 $. We write $ \P $ for $ \P^1 $, $ \jet_x $ for $ \jet_x^1 $, and $ \G_f(x) $ for $  \G_f^{(1)}(x) $.

	\begin{proof}[Proof of Theorem \ref{thm.main}]
		First, we want to show that under the hypotheses of Theorem \ref{thm.main}, the hypotheses of Theorem \ref{thm.Fef06} are satisfied. 
		
		Let $ \ell^* \geq 2n+3 $. Then, $ \G_{\ell^*}(x) $ is a Glaeser fiber, thanks to Lemma \ref{lem.ideal}. By Lemma \ref{lem.stab}, $ \G_{\ell^*}(x) $ is its own Glaeser refinement. Hence, the hypotheses of Theorem \ref{thm.Fef06} are satisfied. 
		
		By Theorem \ref{thm.Fef06}, there exists $ F_0 \in C^1(\Rn) $, not necessarily nonnegative, such that
		\begin{equation}
		\jet_x F_0 \in \G_{\ell^*}(x) \subset \G_0(x) = \G_f(x)\,.
		\label{pjet}
		\end{equation}
		Consider the family of polynomials
		\begin{equation*}
		\mathcal{F} := \set{P^x = \jet_x F_0 : x \in E}\,.
		\end{equation*}
		By Taylor's theorem, $ \mathcal{F} $ satisfies \eqref{eq.wc}. Thanks to \eqref{pjet}, $ \mathcal{F} $ satisfies the hypothesis of Theorem \ref{thm.nw} (with $ m = 1 $). Therefore, there exists $ F \in C^1_+(\Rn) $, such that $ \jet_x F = \jet_x F_0 $ for each $ x \in E $. In particular, $ F(x) = f(x) $. This concludes the proof. 
	\end{proof}

	\begin{remark}
		In \cite{KZ07}, the authors showed that for $ C^1(\Rn) $ without the nonnegative constraint, it suffices to take $ \ksh = 2 $ in the first refinement and $ \ksh = 1 $ in the subsequent refinements, and the number of refinement $ \ell^* $ till stabilization can be reduced to $ n \leq \ell^* \leq n+1 $. It will be interesting to see if these bounds still hold for $ C^1_+(\Rn) $. 
	\end{remark}


	\bibliographystyle{plain}

\begin{thebibliography}{10}
		
		\bibitem{BM07}
		Edward Bierstone and Pierre~D. Milman.
		\newblock $\mathscr{C}^m$-norms on finite sets and $\mathscr{C}^m$-extension
		criteria.
		\newblock {\em Duke Mathematical Journal}, 137(1):1--18, 2007.
		
		\bibitem{BMP03}
		Edward Bierstone, Pierre~D. Milman, and Wiesław Pawłucki.
		\newblock Differentiable functions defined in closed sets. {A} problem of
		{W}hitney.
		\newblock {\em Invent. Math.}, 151(2):329--352, 2003.
		
		\bibitem{BMP06}
		Edward Bierstone, Pierre~D. Milman, and Wiesław Pawłucki.
		\newblock Higher-order tangents and fefferman's paper on {W}hitney's extension
		problem.
		\newblock {\em Ann. of Math. (2)}, 164(1):361--370, 2006.
		
		\bibitem{F05-J}
		Charles Fefferman.
		\newblock A generalized sharp {W}hitney theorem for jets.
		\newblock {\em Rev. Mat. Iberoam.}, 21(2):577--688, 2005.
		
		\bibitem{F05-Sh}
		Charles Fefferman.
		\newblock A sharp form of {W}hitney's extension theorem.
		\newblock {\em Ann. of Math. (2)}, 161(1):509--577, 2005.
		
		\bibitem{F06}
		Charles Fefferman.
		\newblock {W}hitney's extension problem for {$ {C^m} $}.
		\newblock {\em Ann. of Math. (2)}, 164(1):313--359, 2006.
		
		\bibitem{FIL16}
		Charles Fefferman, Arie Israel, and Garving~K. Luli.
		\newblock Finiteness principles for smooth selections.
		\newblock {\em Geom. Funct. Anal.}, 26(2):422--477, 2016.
		
		\bibitem{FIL16+}
		Charles Fefferman, Arie Israel, and Garving~K. Luli.
		\newblock Interpolation of data by smooth non-negative functions.
		\newblock {\em Rev. Mat. Iberoam.}, 33(1):305—324, 2016.
		
		\bibitem{G58}
		Georges Glaeser.
		\newblock Étude de quelques algèbres tayloriennes.
		\newblock {\em J. Analyse Math.}, 6:1--124, 1958.
		
		\bibitem{JL20}
		Fushuai Jiang and Garving~K. Luli.
		\newblock Nonnegative $ {C}^2(\mathbb{R}^2)$ interpolation.
		\newblock {\em To appear in Advance in Math.}
		
		\bibitem{KZ07}
		Bo'az Klartag and Nahum Zobin.
		\newblock $ {C}^1 $ extension of functions and stabilization of glaeser
		refinements.
		\newblock {\em Rev. Mat. Iberoam.}, 23(2):635--669, 2007.
		
		\bibitem{St79}
		Elias M.~Stein.
		\newblock {\em Singular Integrals and Differentiability Properties of
			Functions}, volume~2 of {\em Monographs in harmonic analysis}.
		\newblock Princeton University Press, Princeton, NJ, 1970.
		
		\bibitem{W34-1}
		Hassler {W}hitney.
		\newblock Analytic extensions of differentiable functions defined in closed
		sets.
		\newblock {\em Trans. Amer. Math. Soc.}, 36(1):63--89, 1934.
		
		\bibitem{W34-2}
		Hassler {W}hitney.
		\newblock Differentiable functions defined in closed sets. {I}.
		\newblock {\em Trans. Amer. Math. Soc.}, 36(2):369--387, 1934.
		
		\bibitem{W34-3}
		Hassler {W}hitney.
		\newblock Functions differentiable on the boundaries of regions.
		\newblock {\em Ann. of Math. (2)}, 35(3):482--485, 1934.
		
	\end{thebibliography}

\end{document}